\documentclass[10pt,twocolumn,twoside]{IEEEtran} 
\usepackage{amsthm}
\usepackage[utf8]{inputenc}
\usepackage{amsmath}
\usepackage{amsfonts}
\usepackage{amssymb}
\usepackage{dirtytalk}
\usepackage{graphicx}
\usepackage{comment}
\usepackage[margin=1in]{geometry}
\usepackage{algorithm,algcompatible}
\usepackage{bbm}
\usepackage{xfrac}
\usepackage{dsfont}
\usepackage{hyperref}
\usepackage{algpseudocode}
\usepackage[font=small]{caption}
\usepackage{subcaption}

\usepackage[algo2e]{algorithm2e}
\RestyleAlgo{ruled}
\usepackage[parfill]{parskip}

\DeclareMathOperator*{\argmin}{arg\,min}

\theoremstyle{plain}
\newtheorem{problem}{Problem}
\newtheorem{theorem}{Theorem}
\newtheorem{lemma}{Lemma}
\newtheorem{assumption}{Assumption}

\usepackage{accents}
\newcommand{\ubar}[1]{\underaccent{\bar}{#1}}

\newcommand{\diag}{\text{diag}}
\newcommand{\R}{\mathbb{R}}
\newcommand{\V}{\mathcal{V}}
\newcommand{\G}{\mathcal{G}}

\newcommand{\I}{\mathcal{I}}
\newcommand{\J}{\mathcal{J}}
\newcommand{\lhs}{\partial C}

\newcommand{\param}{\theta^0}

\usepackage[
sorting=none, 
doi=false,
isbn=false,
url=true,
eprint=false,
minbibnames=1,
maxbibnames=1,
useeditor=false
]{biblatex}
\AtNextBibliography{\small}

\title{Decentralized Implicit Differentiation}
\author{
Lucas Fuentes Valenzuela,
Robin Brown,
Marco Pavone
}

\usepackage[dvipsnames]{xcolor}
\begin{document}

\maketitle
\begin{abstract}
The ability to differentiate through optimization problems has unlocked numerous applications, from optimization-based layers in machine learning models to complex design problems formulated as bilevel programs.
It has been shown that exploiting problem structure can yield significant computation gains for optimization and, in some cases, enable distributed computation. 
One should expect that this structure can be similarly exploited for gradient computation.
% One should expect that exploiting problem structure can benefit gradient computation, as it has been shown for pure optimization.
% Despite widespread interest in differentiable optimization, only little attention has been devoted to using the inherent structure present in many optimization programs for efficient gradient computation. 
In this work, we discuss a decentralized framework for computing gradients of constraint-coupled optimization problems.
First, we show that this framework results in significant computational gains, especially for large systems, and provide sufficient conditions for its validity.
Second, we leverage exponential decay of sensitivities in graph-structured problems towards building a fully distributed algorithm with convergence guarantees. 
Finally, we use the methodology to rigorously estimate marginal emissions rates in power systems models. 
% Specifically, we show that the proposed methodology allows to determine the time horizon over which those metrics are important.
Specifically, we demonstrate how the distributed scheme allows for accurate and efficient estimation of these important emissions metrics on large dynamic power system models.
\end{abstract}
\section{Introduction}\label{sec:intro}

The possibility of computing gradients of arbitrary convex optimization programs with respect to their parameters has recently opened the door to many applications. Beyond sensitivity analysis, it allows for gradient-based optimization of an outer objective in bi-level optimization problems, e.g. in hyperparameter optimization~\cite{Bengio2000-fl} and meta-learning~\cite{Franceschi2018-iu}. The representation power of optimization problems has even led researchers to use them as models~\cite{Agrawal2021-vl}, trained via common gradient-based techniques. 

In many cases, e.g. in dynamic programs or network-flow problems, the optimization program can be interpreted as multiple subproblems interconnected through coupling constraints. This structure can be leveraged to design distributed solution methods~\cite{Boyd2011-xn}, which are relevant for enabling parallel computation and reducing the computational burden on individual cores. They are also necessary in contexts where data is distributed and cannot be centralized for privacy reasons as is the case in the federated learning setting~\cite{Kairouz2021-fy}. One should expect that differentiation of a structured optimization program benefits from similar potential for decentralized computation; however, this remains largely unexplored. 
% So far, decentralized method for differentiation of structured optimization programs remain largely unexplored.
In this paper, we propose an efficient framework to compute gradients of separable optimization problems with arbitrary coupling constraints.

\subsection{Related Work}

\paragraph{Differentiable Optimization} Differentiating through an optimization program is usually achieved in two main ways. One method consists of backpropagating through the iteration sequence in an optimization program~\cite{Domke2012-fx}, and is commonly referred to as \textit{unrolling} or \textit{iterative differentiation}. 
It has been used to estimate Stein's Unbiased Risk Estimator for large problem instances~\cite{Nobel2022-ww} and to compute hyperparameter gradients in hyperparameter optimization~\cite{Maclaurin2015-kt,Franceschi2017-ub}, among others.
This approach has the advantage of not requiring an exact solution and permits differentiation at any stage in the solution procedure.
While it is readily compatible with automatic differentiation frameworks, it requires the solver to be differentiable. 
Moreover, storing all computation can be expensive for large numbers of iterations; to address this limitation, truncation methods have been proposed~\cite{Shaban2019-hu}.

In the second approach, gradients of the solution are computed via \textit{implicit differentiation}. While obtaining a closed-form solution is typically not possible, one can represent it implicitly via optimality conditions such as stationarity conditions for unconstrained problems~\cite{Bengio2000-fl}, KKT conditions for Quadratic Programs~\cite{Amos2017-ps,Barratt2018-yk} or self-dual homogeneous embeddings for generic cone programs~\cite{Ye1994-af,Agrawal2019-cr}. This method is attractive because it is agnostic to the solution procedure and is therefore compatible with any type of solver, even those that are not differentiable. 

The flexibility of implicit differentiation has enabled the learning of convex optimization programs representing predictive models~\cite{Agrawal2021-vl} or control policies~\cite{Agrawal2020-tw}. 
It has also allowed the integration of convex optimization programs~\cite{Amos2017-ps,Agrawal2019-fu}, MILPs~\cite{Wilder2019-tl,Ferber2020-dp}, physics simulators~\cite{De_Avila_Belbute-Peres2018-tj} or MPC control policies~\cite{Amos2018-nu} as layers in neural networks, potentially enriching the representational power of these models. Other applications include gradient computation in meta-learning~\cite{Rajeswaran2019-at}, automatic hyperparameter optimization in data-fitting problems~\cite{Barratt2020-wq} and solution refinement in conic solvers~\cite{Busseti2019-gg}.

This approach usually requires an optimal solution and solving a linear system, which can become prohibitive for very large problems. \textit{Approximate} Implicit Differentiation (AID) methods have been developed to deal with this limitation~\cite{Grazzi2020-dg} by approximating the solution to the linear problem via truncated conjugate gradients~\cite{Pedregosa2016-si} or Neumann series~\cite{Lorraine2020-xs}. Recent implementations leverage automatic differentiation to compute gradients through user-defined optimality conditions, thereby allowing the use of any specialized solver of interest~\cite{Blondel2022-tr}.
% A recent study proposes an implementation that captures the benefits of both differentiation strategies~\cite{Blondel2021-tr}: automatic differentiation is used to differentiate through a user-defined function stating the optimality conditions of the problem, thereby allowing the use of any specialized solver of interest. 
% Comparative studies have aimed at assessing the cost of ... \textbf{clarify, elaborate} \cite{Ablin2020-rr, Grazzi2020-ci}.

\vspace{10pt}
% differentiating through an optimization program when it has some structure
\paragraph{Decentralized Differentiation}
The development of distributed optimization has been motivated by the distributed nature of problem data and parameters in many contexts, e.g. in multi-agent systems or in machine learning settings where privacy is a concern.
If those types of problems are to benefit from the numerous applications above, gradient computation needs to be performed in a way that preserves this distributed structure.

Interestingly, only a few recent studies have tackled decentralized gradient computation, usually motivated by bilevel optimization in machine learning settings.
A bilevel program~\cite{Sinha2018-nh} typically consists of two nested optimization problems: the \textit{outer}-level problem depends on a variable which is a solution to the \textit{lower}-level problem. The gradient-based methods common in bilevel optimization require computation of the \textit{hypergradient} (i.e. the gradient of the outer-level problem). When the lower-level problems are unconstrained, which is often the case in machine learning applications, computing the hypergradient necessitates inverting the full Hessian of the objective function, which can be both computationally costly for large-scale problems and difficult to obtain in fully distributed settings. Gradient-tracking methods~\cite{Chen2022-wz, Chen2023-vn, Pu2021-tg} or truncated Neumann series~\cite{Ghadimi2018-jr, Yang2022-rc, Gao2023-yf} are recent suggestions to circumvent those limitations. However, these studies usually consider a restricted set of lower-level problems, where they all share the same optimization variables and are unconstrained. 

% Distributed computation of hypergradients opens the door to distributed bilevel computation. 

\vspace{10pt}
\paragraph{Locality}
Many decision-making problems have a graph as underlying structure. For instance, in dynamic optimization (e.g. optimal control), the graph is linear and links two consecutive timesteps together. In multi-stage stochastic programming, the graph corresponds to a scenario tree, while in network optimization the graph represents a physical network.
This problem structure can be actively exploited for efficient solution methods. Indeed, in \textit{graph-structured} problems, the sensitivity of a subproblem with respect to another subproblem's parameters decreases as the distance on the graph between those problems increases. Rebeschini and Tatikonda~\cite{Rebeschini2019-td} proposed a notion of correlation between variables in network optimization and show that this metric decays exponentially as a function of the graph distance on the network. This so-called \textit{locality} has been used to develop efficient solution methods for multi-agent optimization problems~\cite{Brown2020-ad, Brown2021-ju}. While these works focus on a particular definition of correlation, Shin \textit{et al.} proved that \textit{exponential decay in sensitivities} is a general feature of graph-structured optimization problems~\cite{Shin2022-jo} and developed decentralized solution schemes that rely on this property~\cite{Shin2020-jc}.

\subsection{Statement of Contributions and Organization}

% In this paper, we derive a decentralized framework for differentiating through separable optimization programs with arbitrary constraints. By using a dual decomposition approach, we separate the hypergradient into local gradients and a correction term emanating from the coupling between subproblems.
% We develop sufficient conditions for all gradients in the computation to be unique and well-defined. 
% In addition, we demonstrate that graph-structured problems are well suited to our framework. 
% Namely, we take advantage of the exponential decay in sensitivities to enable fully distributed computation of the coupling term.  
% Finally, we illustrate the benefits of the proposed algorithm on the computation of marginal emissions rates in power system models. 

In this paper, we develop a decentralized framework for differentiating through separable optimization problems with arbitrary coupling constraints.
The approach provides a significant computational advantage, particularly when the number of coupling constraints is small relative to the size of the entire problem.
We extend these results further in the case of graph-structured constraints. By leveraging exponential decay of sensitivities, we enable fully distributed differentiation with convergence guarantees.

This document is structured as follows. In Section~\ref{sec:DID}, we derive the differentiation framework, analyze its computational complexity and provide sufficient conditions for the validity of the approach. 
In Section~\ref{sec:distributed}, we generalize the analysis of existing distributed schemes~\cite{Shin2020-jc} and demonstrate that these algorithms can be applied, with convergence guarantees, to our problem setting. In Section~\ref{sec:experiments}, we validate the computational gains and the fully decentralized scheme via simulation on randomly generated problems. In Section~\ref{sec:MEF_approx}, we apply the fully decentralized scheme to the computation of marginal emissions in power systems models.
% In Section~\ref{sec:applications}, we illustrate the benefits of decentralized implicit differentiation on resource allocation and hyperparameter optimization. 
We conclude in Section~\ref{sec:conclusion}.

% This enables efficient computation of gradients for large and weakly coupled problems, and enables decentralized bi-level computation for this large class of optimization programs. 

\subsection{Notation}
% The set of real numbers and of integers are denoted by $\R$ and $\mathbbm{I}$, respectively. 
% We define $\mathbbm{I}_{A}:= \mathbbm{I} \cap A $, where $A$ is a set.
We denote by $e_i$ the canonical $i$th basis vector and $I_n$ the identity matrix of dimensions $n \times n$. For a set of indices $\mathcal{I} = \{i_1, ..., i_N\}$, we define $\{v_i\}_{i \in \mathcal{I}} = [v_{i_1},..., v_{i_N}] $.
Consider a multivariate function $F(x) : \R^M \rightarrow \R^N$.
We let $DF(x) \in \R^{N \times M}$ denote the Jacobian matrix of $F$ evaluated at $x$; the entry $DF_{nm}(x)$ is the derivative of $F_n$ with respect to $x_m$.
For functions with multiple arguments $F(x, y) : \R^N \times \R^L \rightarrow \R^L$, we use $\partial_x F(x, y) \in \R^{L \times N}$ and $\partial_y F(x, y) \in \R^{L \times L}$ to denote the partial Jacobians of $F$ with respect to its first and second arguments, respectively.

\section{Decentralized Implicit Differentiation}\label{sec:DID}

In this section, we lay out the generic differentiation framework proposed in this paper, discuss the associated computational gains and provide conditions for the validity of the approach.

\subsection{Problem Statement}\label{sec:problem_statement}

We consider \textit{constraint-coupled problems} consisting of $N$ subproblems coupled via both equality and inequality constraints~\cite{Falsone2020-pr, Su2022-kl}. We will assume the coupling constraints are themselves separable, i.e. they can be written as a sum where each term depends on a single subproblem. This is the case for any affine constraint. 
\begin{problem}[Separable constraint-coupled problem]
\begin{align*}
\begin{array}{ll}
    \min_{\mathrm{x}} & \sum_{i=1}^N f_i^0(x_i, \theta_i),\\
    \textrm{s.t. }& x_i \in \mathcal{C}_i(\theta_i), \ i = 1, ...,\ N, \\ 
     % & \mathrm{x} \in \mathcal{D}(\theta_{\mathcal{D}})
% \textrm{and s.t.}   
&  \sum_{i=1}^N H_i(x_i, \theta_c) = 0,\\
   & \sum_{i=1}^N F_i(x_i, \theta_c) \leq 0.
\end{array}
\end{align*}
\label{pbm:CCOP}
\end{problem}
\vspace{-10pt}
We will refer to this generic problem as the \textit{global} or \textit{centralized} problem.
The $N$ local convex optimization problems, indexed by $i \in \{1, \ldots N\}$, are defined by the local variables, $x_i \in \mathbbm{R}^{n_i}$, objectives $f_i^0: \mathbbm{R}^{n_i} \times \mathbbm{R}^{t_i} \to \mathbbm{R}$, and constraint sets $\mathcal{C}_i(\theta_i) \subseteq \mathbbm{R}^{l_i}$.
Both the objectives, $f_i^0(\cdot, \theta_i)$, and constraint sets, $\mathcal{C}_i(\theta_i)$, depend on local parameters $\theta_i \in \mathbbm{R}^{t_i}$ and are specific to each subproblem.
% In the above formulation, the global objective is separable in $N$ individual objectives with $f_i^0: \mathbbm{R}^{n_i} \times \mathbbm{R}^{t_i} \to \mathbbm{R}$ denoting convex local objective functions of both the local primal variables $x_i \in \mathbbm{R}^{n_i}$ and local parameters $\theta_i \in \mathbbm{R}^{t_i}$. The local constraint sets $\mathcal{C}_i(\theta_i) \subseteq \mathbbm{R}^{l_i}$ are assumed convex. 
All $N$ subproblems are jointly constrained by $\Lambda_H$ coupling equality constraints, $H_i(x_i, \theta_c) : \R^{n_i} \times \R^{t_c} \to \R^{\Lambda_H}$, and $\Lambda_F$ coupling inequality constraints, $F_i(x_i, \theta_c) : \R^{n_i} \times \R^{t_c} \to \R^{\Lambda_F}$,
% $H(\mathrm{x}) \subseteq \mathbbm{R}^{\Lambda_H}$ and $F(\mathrm{x}) \subseteq \mathbbm{R}^{\Lambda_F}$,  
% is constrained to a convex set
where $\theta_{c} \in \mathbbm{R}^{t_c}$ parameterizes those coupling constraints. Finally, we will denote the vector of all primal variables by $\mathrm{x} = [x_1^T, ..., x_N^T]^T$ and all parameters  by $\theta = [\theta_1^T, ..., \theta_N^T, \theta_{c}^T]^T $.

The aim of this paper is to develop a method that leverages the inherent structure in the above problem to compute the Jacobian $D_{\theta}x_i^\star$, for all $i$. 
% We will assume the coupling equality and inequality coupling constraints are \textit{separable}, i.e.
% \begin{equation}
% \label{eq:coupling_constraints}
%     \begin{array}{lll}
%          & H(\mathrm{x}, \theta_c)=\sum_i H_i(x_i, \theta_c) & = 0,\\
%          & F(\mathrm{x}, \theta_c)=\sum_i F_i(x_i, \theta_c) & \leq 0.
%     \end{array}
% \end{equation}
Associating dual variables $\nu \in \mathbbm{R}^{\Lambda_H}$ and $\lambda \in \mathbbm{R}^{\Lambda_F}$ to the equality and inequality coupling constraints, respectively, we will assume we have a solution $(\mathrm{x}^\star(\param), \nu^\star(\param), \lambda^\star(\param))$ to the optimization problem for parameters $\param$ via any appropriate solution method. The problem structure is particularly amenable to distributed algorithms ~\cite{Falsone2020-pr}.
We also make the following technical assumptions.
% , and includes many practical problems as discussed in Section~\ref{sec:applications}. 
% While we consider here inequality coupling constraints defined on the non-negative cone, derivations still apply to other cones. This topic is explored further in Appendix~\ref{appendix:self-dual}. 
% In the case where the inequality constraints are not straightforwardly separable, the main idea still applies despite slightly more tedious calculation (see Appendix~\ref{appendix:non-separable}).

\begin{assumption}[Differentiability]
The functions defining the problem are twice continuously differentiable with respect to both $\mathrm{x}$ and the parameters $\theta$ in a neighborhood of $\mathrm{x}^\star(\param)$.
\label{A1}\end{assumption}

\begin{assumption}[First order conditions] The solution $(\mathrm{x}^\star(\param), \nu^\star(\param), \lambda^\star(\param))$ satisfies the first order KKT conditions.
% , i.e.
% \begin{align*}
% \begin{array}{rcl}
%      g_i(x^\star, \param) & \leq   0, &  i=1, ..., l\\
%      h_j(x^\star, \param) & = 0, & j=1, ..., m\\
%      \lambda_i^\star g_i(x^\star, \param) &= 0, & i = 1, ..., l\\
%      \lambda_i &\geq 0, & i = 1, ..., l,\\
%    \nabla_x L(x^\star, \lambda^\star, \nu^\star, \param)  &=  0, &
% \end{array}
% \end{align*}
% % \vspace{-15pt}
% where $L$ is the Lagrangian associated to~\eqref{eq:validity_P}.
\label{A2}\end{assumption}

\begin{assumption}[Second order conditions]
The solution to Problem~\ref{pbm:CCOP} satisfies the second order optimality conditions $u^T \nabla^2 L(\mathrm{x}^\star, \lambda^\star, \nu^\star, \param) u > 0, \ \forall u \neq 0$ s.t.
\begin{align*}
& u^T \nabla g_i(\mathrm{x}^\star, \param) \leq 0, \text{ for all } i \text{ where } g_i(\mathrm{x}^\star, \param) = 0, \\ 
& u^T \nabla g_i(\mathrm{x}^\star, \param) = 0, \text{ for all } i \text{ where } \lambda^\star_i > 0, \\
& u^T \nabla h_j(\mathrm{x}^\star, \param) = 0,\  j = 1, ..., m,
\end{align*}
where $g_i\ (i = 1, ..., n)$ and $h_j\ (j=1, ..., m)$ represent \emph{all} the inequality and equality constraints of the global problem---including the local constraints--- respectively, and $L$ denotes the Lagrangian.
% \vspace{-15pt}
\label{A3}\end{assumption}
Assumption~\ref{A3} is always satisfied if the problem is strictly convex. However, for linear programs (LPs), it requires that the solution $\mathrm{x}^\star$ lie on a vertex of the constraint set, and that all the binding constraints be linearly independent. Those conditions are sufficient to ensure that, locally, a minimizing point is unique.
% as stated by the following lemma.
% \begin{lemma}
% If Assumption~\ref{A1} is satisfied, and if there exist Lagrange multipliers $\lambda^\star \in \mathbb{R}^l, \nu^\star \in \mathbb{R}^m$ such that Assumptions~\ref{A2}-~\ref{A3} are satisfied, then $x^\star$ is a local isolated (locally unique) minimizing point of \textbf{P}.
% \label{lemma:x_unique}\end{lemma}
% In addition, if the functions defining the problem are convex, then a local minimizer is also a global minimizer.
Note that while the first order KKT conditions are sufficient to ensure that the problem has been solved to optimality, they do not imply the uniqueness of the solution point. 
% For the Jacobian to be invertible, it is critical that the solution be unique, in addition to two more properties. 

\begin{assumption}[Linearly independent binding constraints] A solution to the optimization problem is such that the gradients of the binding constraints are linearly independent.
\label{A4}\end{assumption}

\begin{assumption}[Strict complementary slackness] The dual variable associated to inequality constraints is strictly positive for every binding inequality constraint.
\label{A5}
\end{assumption}

\subsection{Differentiation Framework}\label{sec:decentralized_diff}
The main idea in this paper consists of decomposing the Jacobian $D_{\theta}x_i^\star$ in a local component
% obtained by local differentiation 
and a correction term emanating from the coupling between all subproblems. 
The former is obtained by differentiating each subproblem independently, while the latter results from differentiating through the KKT conditions associated with the coupling constraints. Combining both components yields the expected total sensitivity, and accounts for both local and coupling effects across subproblems.
% The former is obtained by applying implicit differentiation to each subproblem considered independently, and yields the Jacobian of the local optimal solution with respect to both the local parameters and the coupling dual variables. Then, we differentiate the KKT conditions associated with the coupling constraints to obtain the Jacobian of the coupling dual variables with respect to the parameters. Combining both components yields the expected total sensitivity, and accounts for both local and coupling effects across subproblems.

Precisely, we construct the Lagrangian by dualizing only the coupling constraints as
\begin{align*}
    L(\mathrm{x}, \theta, \nu, \lambda)
    = \sum_i^{N} &\left(f_i^0(x_i, \theta_i)\right. \\
    &\left. +\ \nu^TH_i(x_i, \theta_c) + \lambda^T F_i(x_i, \theta_c)\right).
\end{align*}
% To lighten notation, we will drop explicit dependence on $\theta_c$ in the remainder of this text as we focus on computing $D_{\theta_i} x_i^\star$.
Given the optimal dual variables, the optimal primal solution $\mathrm{x}^\star$ satisfies the following equation
\begin{equation*}
    \mathrm{x}^\star = \argmin_{x_i \in \mathcal{C}_i(\theta_i),\ i=1,...,N} L(\mathrm{x}, \theta, \nu^\star, \lambda^\star).
\end{equation*}
As the Lagrangian is itself separable, we can optimize each subproblem independently, 
\begin{equation*}
\label{eq:subpb_xi}
    x_i^\star = \argmin_{x_i \in \mathcal{C}_i(\theta_i)} \ f_i^0(x_i, \theta_i) + \nu^{\star T}H_i(x_i, \theta_c) + \lambda^{\star T} F_i(x_i, \theta_c).
\end{equation*}
% The above expression refers to what we will call a \textit{local} problem or a \textit{subproblem}. 
While subproblem $i$ explicitly involves only the parameters $\theta_i$ and $\theta_c$, it is implicitly dependent on all parameters $\theta$ via the coupling constraints. This dependence can be made explicit by expressing the optimal solution as $x_i^\star = x_i^\star(\theta_i, \theta_c, \nu^\star(\theta), \lambda^\star(\theta))$.
% [The local component correspond to computing the Jacobians of the local optimal solutions ,$x_i^*$, with respect to $\theta_i, \theta_c, \nu^*$ and $\lambda^*$. The correction term comes from applying the implicit function term to deduce the sensitivities of $\nu^*, \lambda^*$ with respect to $\theta$. Finally the sensitivities are computed by applying the chain rule.]
Our key insight is that when the dependence of $\nu,\ \lambda$ on  $\theta$ is ignored, their interpretation as dual variables can be temporarily disregarded.
This allows us to treat them as independent \textit{parameters} in each subproblem.
The local component in decentralized differentiation then consists of computing the Jacobians of local optimal variables $x_i^\star$ with respect to $\theta_i,\ \theta_c,\ \nu^\star$ and $\lambda^\star$. The correction term, on the other hand, comes from applying the implicit function theorem to deduce the sensitivities of $\nu^\star,\ \lambda^\star$ with respect to $\theta$. Finally, the total sensitivities are computed by applying the chain rule.
% Instead of considering the dual variables $\nu, \lambda$ as variables, we consider them instead as \textit{parameters} in each subproblem. That is, we temporarily disregard the condition that those dual variables need to satisfy. 

Specifically, we define an augmented local parameter vector $\bar{\theta}_i = [\theta_i, \theta_c, \nu, \lambda]$ that includes both the original parameters and local copies of the dual variables. This disentangles all subproblems for the purposes of differentiation.
Let $z_i$ denote the vector of local primal and dual variables and $G_i(z_i, \bar{\theta}_i) = 0$ the KKT conditions that $z_i$ must satisfy at optimality~\cite{Barratt2020-wq, Amos2017-ps}.
We can then compute the \textit{local Jacobian} with respect to the augmented local parameters via the implicit function theorem~\cite{Dontchev2009-ep}
\begin{align}\label{eq:local_jac}
    \partial_{\bar{\theta}_i}x_i^\star 
    &=  - \begin{bmatrix}
        I_{n_i} & 0
    \end{bmatrix}\left(\partial_{z_i} G_i \right)^{-1}\partial_{\bar{\theta}_i} G_i.
    % &= - \begin{bmatrix}
    %     I_{n_i} & 0
    % \end{bmatrix} \left(\partial_{z_i} G_i \right)^{-1} [\partial_{\theta_i} G_i, \partial_{\nu} G_i, \partial_{\lambda} G_i].
\end{align}
% or a self-dual homogeneous embedding~\cite{Ye1994-af}.
% In summary, all subproblems are considered independent and the coupling constraints are not explicitly enforced. Instead, they are incorporated in the objectives of the subproblems and implicitly enforced via the value of the dual variables. 
% For the purpose of derivation, those dual variables are only considered parameters, which allows the computation of derivatives with respect to both local parameters and dual variables. 
We note that~\eqref{eq:local_jac} only depends on local information, and thus can be computed in a fully distributed fashion, without any exchange of information between subproblems.

While the above derivations treat $\nu^\star,\ \lambda^\star$ as fixed parameters, in reality they depend on the entire parameter vector $\theta$.
% When the global problem is solved to optimality, the coupling constraints are naturally satisfied. 
% Therefore, a change in parameter $\theta$ implies a potential change in the dual variables $\nu, \lambda$, which will impact the optimal primal values of potentially any subproblem.
Therefore, a change in any parameter $\theta_j$ of subproblem $j$ may impact the optimal value of any other subproblem $i\neq j$.
% Intuitively, one expects coupling to also be present between individual gradients. 
Indeed, the subproblems are related through the KKT conditions associated with the coupling constraints: 
\begin{equation}
    C\left(\mathrm{x}^\star(\theta)
    % (\theta, \nu^\star(\theta), \lambda^\star(\theta))
    , \lambda^\star(\theta), \theta_c \right) = 
    {\scriptsize
    \begin{bmatrix}
    \sum_iH_i(x_i^\star, \theta_c)\\
    \diag(\lambda^\star)\sum_iF_i(x_i^\star, \theta_c)
    \end{bmatrix}}
    = 0.
\end{equation}
Differentiating this expression with respect to $\theta$ and using $D_\theta C=0$ results in the following linear system:
\begin{equation}
\label{eq:coupling_gradients}
    % \scriptsize
    \begin{array}{rl}
    \lhs  \begin{bmatrix}
        D_\theta \nu \\
        D_\theta \lambda
    \end{bmatrix} 
    &= q, \\
    \lhs 
    % = \sum_i \lhs_i  
    &= -\begin{bmatrix}
        \partial_\mathrm{x} C \partial_\nu \mathrm{x}
        & \partial_\mathrm{x} C \partial_\lambda \mathrm{x} + \partial_\lambda C
    \end{bmatrix},\\
    q 
    % = \sum_i q_i 
    &= \begin{bmatrix}
        \partial_\mathrm{x} C \partial_\theta \mathrm{x} + \partial_\theta C
    \end{bmatrix}.
    \end{array}
\end{equation}
The above equation expresses that local gradients are coupled across subproblems, similarly to primal variables. We will refer to the variable $y := [D_\theta \nu^T \ D_\theta \lambda^T]^T$ as the \textit{coupling Jacobian}.
Computing the coupling Jacobian therefore requires solving a system of equations involving local gradients from all subproblems.
In a federated setting, a central computing node can aggregate local gradients and solve~\eqref{eq:coupling_gradients}. In Section~\ref{sec:distributed}, we discuss a fully distributed scheme to solve~\eqref{eq:coupling_gradients}.

The separability of Problem~\ref{pbm:CCOP} makes~\eqref{eq:coupling_gradients} decomposable in local terms. Indeed, it is equivalent to
\begin{align}
    \sum_i \lhs_i
    \begin{bmatrix}
        D_\theta \nu\\
        D_\theta \lambda
    \end{bmatrix}
     = 
     \sum_i q_i ,
        \label{eq:coupling_gradients_sum}
\end{align}
where\footnote{We write $H_i = H_i(x_i^\star, \theta_c)$ (and similarly for $F_i$) to lighten the notation.}
% the local left-hand sides $\lhs_i$ are written as
\begin{align*}
    \scriptsize
    \lhs_i = 
    -\begin{bmatrix}
    \partial_{x_i} H_i \partial_\nu x_i & \partial_{x_i} H_i \partial_\lambda x_i \\
        \diag(\lambda^\star)\partial_{x_i} F_i\partial_\nu x_i & \diag(\lambda^\star)\partial_{x_i} F_i \partial_\lambda x_i + \diag(F_i)
        \end{bmatrix}
\end{align*}
and
% the local right-hand sides
$q_i = 
    \scriptsize
        \begin{bmatrix}
        \partial_{x_i} H_i\partial_\theta x_i+ D_\theta H_i \\ 
        \diag(\lambda^\star)(\partial_{x_i} F_i\partial_\theta x_i  + D_\theta F_i)
        \end{bmatrix}.
$

% Now that we have derived the local Jacobians and coupling Jacobians,
To obtain the sensitivity of optimal solutions with respect to parameters, one only needs to combine the local and coupling Jacobians for all subproblem  via the chain rule 
\begin{align}
   D_{\theta}x_i^\star(\theta, \nu, \lambda) = \partial_{\theta} x_i^\star + \partial_{\nu}x_i^\star D_{\theta}\nu  + \partial_{\lambda} x_i^\star D_{\theta}\lambda,
   \label{eq:update_local}
\end{align}
where $\partial_{\theta} x_i^\star$ is non zero only for the blocks corresponding to the local parameters $\theta_i$ and the coupling parameters $\theta_c$. We see that the Jacobian naturally decomposes into a local and a coupling component. The former can be associated with the direct impact of a parameter change on the optimal variable, while the latter reflects an indirect impact. We note that the approach proposed here can be interpreted as a block inversion of the global Jacobian. Indeed, the coupling Jacobian $\lhs$ in~\eqref{eq:coupling_gradients} is the Schur complement of the block diagonal matrix whose blocks are the local Jacobians.

\subsection{Conditions for validity}
We now present conditions for the proposed differentiation scheme to be well-defined. 
% We first consider a generic optimization problem of the form \textbf{P}
% \begin{align}
% \begin{array}{ll}
%       & \min_x f(x, \theta)\\
%     \text{s.t. } & g_i(x, \theta) \leq 0,\ i=1, ..., l,\\
%     & h_j(x, \theta) = 0, \ j=1,..., m,
% \end{array}
% \label{eq:validity_P}
% \end{align}
% and assume we have a solution $x^\star(\param) \in \mathbbm{R}^n$ with associated dual variables $\lambda^\star(\param) \in \mathbbm{R}^l,\ \nu^\star(\param) \in \mathbbm{R}^m$ for parameter values given by $\param$. 
Conditions for the Jacobian of the solution to a convex optimization program to exist and be invertible have been detailed in previous work~\cite{Fiacco1976-bt}. Instead of evaluating a single Jacobian for the full problem, the approach proposed in Section~\ref{sec:decentralized_diff} relies on evaluating $N+1$ Jacobians, all of which need to exist and be invertible. Building on results in~\cite{Fiacco1976-bt}, we state sufficiency conditions for all Jacobians involved in the framework of this paper to exist, be unique and invertible.

\begin{theorem}[Decentralized implicit differentiation]
Consider a constraint-coupled optimization problem in the form of Problem~\ref{pbm:CCOP}.
If Assumptions~\ref{A1},~\ref{A2},~\ref{A4},~\ref{A5} hold for the global problem and if Assumption~\ref{A3} holds for each subproblem individually, then all $N$ local Jacobians and the coupling Jacobian in decentralized implicit differentiation exist, are unique and invertible.
\label{th:DID}\end{theorem}
The proof is provided in Appendix~\ref{app:proof_did}.
The main difference between Theorem~\ref{th:DID} and the original result in~\cite{Fiacco1976-bt} is that \textit{local} subproblems also need to satisfy the second order conditions. This is because these do not necessarily hold locally, even if they hold globally. As an illustration, consider the problem 
$\min_{x_1, x_2} \frac{1}{2}(x_1 - \theta)^2,\ 
    \text{s.t. } x_1 = x_2.$
This problem satisfies the second order conditions. Indeed, we have $\nabla^2 L = \scriptsize \begin{bmatrix}
    1 & 0\\ 0 & 0
\end{bmatrix}, $ and $u^T \nabla^2L u > 0,$ for all $ u = \alpha [1, 1]^T,\ \alpha \in \mathbb{R}_0$. However, the local second order conditions for the second subproblem $\min_{x_2} 0 - \nu x_2$ is $0 > 0$, which cannot be true. Therefore, for this example, decentralized implicit differentiation would not be applicable as the second local Jacobian would be singular.
% what is the implication of the above for linear programs?

\textbf{\textit{Remark}:} If a subproblem is a Linear Program (LP) with primal variable $x_i \in \R^{n_i}$, Assumption~\ref{A3} implies that $x_i$ lies at the vertex of $n_i$ \textit{local} constraints. Therefore, that subproblem cannot be involved in additional tight global coupling constraints as it would otherwise be overconstrained. While gradients would still be well-defined, we will not discuss such examples as the subproblem $i$ would virtually be uncoupled from the others. In the remainder of this text, we will therefore assume each subproblem to be \textit{strictly convex}.

\subsection{Complexity analysis}\label{sec:complexity}
% todo: clarify the notations of the dimensions, it's a bit messy right now
Both the centralized and decentralized differentiation schemes require solving linear systems (see~\cite{Dontchev2009-ep} and Eqs.~\eqref{eq:local_jac},~\eqref{eq:coupling_gradients}).
In general, solving a linear system of dimension $m$ requires $O(m^3)$ operations. For a problem in the form of Problem~\ref{pbm:CCOP}, with $N$ subproblems each with $n$ local variables and $l$ local constraints, and with $\Lambda$ coupling constraints, centralized gradient computation requires $O((N(n+l)+\Lambda)^3)$ operations. On the other hand, the decentralized approach decomposes it in $N$ problems of size $n+l$ and one problem of dimension $\Lambda$. In that case, the complexity is $O(\Lambda^3 + N(n+l)^3)$. Therefore, decentralized computation can prove advantageous whenever $\Lambda$ is small compared to other dimensions in the system. 
We define the \textit{coupling ratio} $\rho = \frac{\Lambda}{N(n+l)}$ as a measure of how strongly coupled all the subproblems are with each other. Intuitively, one expects the computational benefit of the decentralized computation scheme to grow as $\rho$ decreases. Indeed, we can write the ratio of expected compute times as
\begin{align}
    \eta = \frac{\Lambda^3 + N(n+l)^3}{(\Lambda+N(n+l))^3} = \frac{\rho^3 + 1/N^2}{(1+\rho)^3}.
    \label{eq:complexity}
\end{align}
From this equation, we deduce that when the number of subproblems $N$ is large and $\rho$ is small, the ratio of compute times grows as $\rho^{3}$. In settings where the coupling between subproblems is weak (i.e. $\rho$ is small), one can therefore expect significant speedup compared to centralized computation. Let us note that the above considerations do not factor in the possibility to conduct local operations in parallel. In this case, total compute time could be further decreased by computing the local gradients concurrently.
% maybe put the illustration of computational gains in Section 2?
% \input{body/3.Distributed}
\section{Distributed Implicit Differentiation in Graph-Structured Problems}\label{sec:distributed}

% Intro

In graph-structured problems, the graph explicitly models the coupling relationships between subproblems. Such structure results in \textit{locality}, i.e. an exponential decay of sensitivities. In this section, we discuss how this feature of graph-structured problems allows for a fully distributed scheme for implicit differentiation, and provide associated theoretical guarantees.
% Namely, we adapt the synchronous decentralized algorithm from~\cite{Shin2020-jc} to solve the coupling system in~\eqref{eq:coupling_gradients} in a fully distributed fashion.
% We first describe the main elements in the initial algorithm and the adaptations to our setting. 
% Then, we give theoretical guarantees on its convergence.

% We first recall what exponential decay of sensitivities means in our context (implications).
% Then, we describe the main elements in the algorithm from Shin, and its adaptation to our case. Third, we give theoretical guarantees on the convergence of this algorithm

\subsection{Distributed solution to coupling system}

The scheme derived thus far requires a central computing node to solve the coupling problem in~\eqref{eq:coupling_gradients}, as it involves gradient information from all subproblems.
However, locality results imply that, in a graph-structured optimization program, the sensitivity of the optimal variables $x_k^\star$ of subproblem $k$ with respect to parameters $\theta_j$ of subproblem $j$ decays exponentially with the distance between subproblems $k$ and $j$ in the graph---the more distant they are, the smaller the sensitivity of one with respect to the other's parameters. If each subproblem is attributed to an individual computing node, one might expect that a reasonable local approximation of the coupling problem can be obtained by considering information from the nearest nodes in the graph only.

In this section, we formalize this intuition and propose a distributed scheme---adapted from the \textit{synchronous coordination} scheme in~\cite{Shin2020-jc}---for gradient computation where each node solves a local approximation of the original problem, iteratively. Those approximations are effectively local projections of the original problem, and simply consist of an appropriate selection of rows and columns corresponding to subsets of the coupling constraints in the original system.

In order to describe the distributed algorithm, we first need to define some terminology specifying the restriction of the system in Equation \eqref{eq:coupling_gradients} to subsets of its rows and columns. 
Formally, we model the structure of Problem~\ref{pbm:CCOP} as a bipartite graph $\G_b$, whose nodes are split between those that represent the subproblems, $\mathcal{V}_p$ ($|\mathcal{V}_p| = N$), and those that represent the constraints, $\mathcal{V}_c$ ($|\mathcal{V}_c| = \Lambda_H + \Lambda_F$). An edge $(k,j),\ k \in \V_p,\ j\in\V_c,$ is present if subproblem $k$ is directly involved in constraint $j$.
For a given subproblem $k$, the local projection of the coupling problem is obtained by keeping the rows and columns in~\eqref{eq:coupling_gradients} corresponding to constraints that are within a given distance 
% $\omega \in \mathbb{Z}$ 
from node $k$ in $\G_b$.
% , where $\omega$ is a tunable parameter dictating the accuracy of the decentralized scheme.
% are sufficiently close to $k$ in the graph.
% Each subproblem is linked to others indirectly via constraints. This graph directly induces structure on $\lhs$.
% The maximum distance considered is quantified by the horizon parameter $\omega \in \mathbb{Z}$. 
% Each subproblem $k \in \mathcal{V}_p$ is thus associated with 
Specifically, we associate to node $k$ a restricted set of constraint nodes in the bipartite graph
$$
\mathcal{V}_k^\omega = \{c \in \mathcal{V}_c \ |\  d_{\G_b}(c, k) \leq 2\omega +1\},
$$
% The parameter $\omega \in \mathbb{Z}$ quantifies how far in the graph the constraints considered are from a given subproblem. 
where $ d_{\G}(i, j)$ is the distance between nodes $i, j$ on a graph $\G$ and the parameter $\omega \in \mathbb{Z}$ sets the distance of interest. 
For instance, $\mathcal{V}_k^0$ is the set of constraints that subproblem $k$ is directly involved in. In the example in Fig.~\ref{fig:graph_illustration}, subproblem 1 is one hop away from constraint 5, so $\V_1^0 = \{5\}$. However, the large connectivity of subproblem 2 results in $\V_1^{1} =  \{5, 6, 7\} = \V_2^{0}$.
% If the coupling constraints in~\eqref{eq:coupling_constraints} are structured and sparse, each subproblem in $\V_p$ will be linked to only a few constraints. Therefore, the Jacobians $\partial_\nu x_i$ and $\partial_\mu x_i$ will also be sparse. Namely, they will be nonzero only at the columns corresponding to the constraints subproblem $i$ is directly involved in, i.e. $\V_k^0$. Subproblem $k$ thus only requires the coupling Jacobians $D_\theta \nu, D_\theta \lambda$ for the constraints in $\V_k^0$.

\begin{figure}
    \centering
    \includegraphics[width=.25\textwidth]{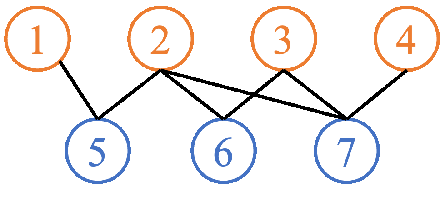}
    \caption{Illustrative bi-partite graph linking subproblems in orange and coupling constraints in blue.}
    \label{fig:graph_illustration}
    \vspace{-10pt}
\end{figure}

These sets are used to define projection operators for each subproblem $k$ and value of $\omega$ as follows:
(a) $T_k := \{e_i \}_{:, i\in \mathcal{V}_k^0 }$ projects on the set of constraints that subproblem $k$ is directly involved in; 
(b) $T_k^\omega := \{e_i \}_{:, i\in \mathcal{V}_k^\omega }$ projects on the set of constraints that are within a distance of $2\omega+1$ away from subproblem $k$; 
(c) $T_{-k}^\omega := \{e_i \}_{:, i\in \mathcal{V}_c\setminus \mathcal{V}_k^0 }$ projects on the set of constraints that are beyond a distance of $2\omega+1$ from subproblem $k$.
These projection operators are used to define
\begin{align}
\label{eq:projections}
\begin{array}{rl}
    \lhs^\omega_k &:= (T^\omega_k)^T \lhs T^\omega_k,\\ 
    \lhs^\omega_{-k} &:= (T^\omega_k)^T \lhs T^\omega_{-k},\\
    q^\omega_k &:= (T_k^\omega)^Tq.
\end{array}
\end{align}
% Intuitively, these matrices are the projections of the original system onto the set of constraints specified by $\mathcal{V}_k^\omega$. In particular, ...
These matrices correspond to sub-matrices obtained by selecting appropriate rows and columns.
% associated with nodes in $\V_k^\omega$ or $\V_{-k}^\omega$.
Note that constructing them does not require constructing the full $\lhs$ matrix or $q$ vector. Indeed, $\lhs,\ q$ are a sum of local components as shown in~\eqref{eq:coupling_gradients_sum}, and only a limited number of terms are required to construct a given local projection.\footnote{Rigorously, each subproblem $k$ collects data from nodes in $\mathcal{P}_k^\omega = \{ i \in \mathcal{V}_p \ | \ d_\G(i, \mathcal{V}_k^\omega) \leq 1\}$, the set of subproblems that are directly involved in at least one of the constraints in $\mathcal{V}_k^\omega$.}

Now that we have defined these restricted systems, we are in a position to present the distributed algorithm for approximating the coupling Jacobian (i.e., finding $\hat{y}$ such that $\partial C \hat{y} \approx q$).
The objective, however, is not for each computing node to estimate the entire coupling Jacobian; 
% instead, at iteration $t$, the $k$-th computing node maintains an estimate $\hat{y}^{(t)}_k = \{\hat{y}_i\}_{i \in \V^0_k}$ restricted to the constraints it directly participates in.
instead, at iteration $t$, the $k$-th computing node maintains an estimate of only the components of $\hat{y}^{(t)}$ corresponding to the constraints in $\V_k^0$ (i.e. those it directly participates in), 
which we denote $\hat{y}^{(t)}_k$.
% In this section, we will write the entire vector for ease of notation, however all operations are implicitly localized based on these restricted subsets. 
In the following, all operations are localized based on these restricted subsets. 
This is sufficient for computing Equation \eqref{eq:update_local}, as the local Jacobian with respect to the dual coupling variables $\partial_\nu x_k^\star,\ \partial_\lambda x_k^\star$ are nonzero only for the coupling constraints in $\V^0_k$. 

At the beginning of the algorithm, it is assumed that each node has computed its local Jacobian according to Equation \eqref{eq:local_jac}.
The first estimate $\hat{y}^{(0)}$ is initialized at random and shared across the network. Each node builds the above matrices and updates its local estimate of the coupling Jacobian according to 
\begin{align}
\label{eq:iteration}
    \hat{y}_k^{(t+1)} = S_k^\omega \hat{y}^{(t)} + U_k^\omega q.
\end{align}
where
\begin{align}
    S_k^\omega &:= - (T_k)^T T_k^\omega (\lhs_k^\omega)^{-1} \lhs_{-k}^\omega (T_{-k}^\omega)^T \label{eq:Sk_def}, \\
    U_k^\omega & := (T_k)^TT_k^\omega(\lhs_k^\omega)^{-1}(T_k^\omega)^T \label{eq:Uk_def}.
\end{align}

% While the right hand side is expressed in terms of the entire vectors $y^{(t)}$ and $q$, $S_k^\omega$ and $U_k^\omega$ map these vectors to the components corresponding to $\V^\omega_k$. The final multiplication 
% This iteration can be performed in parallel for every subproblem.
% The variable $\omega$ controls how far from $k$ in the graph information is gathered. We note that this iteration can be performed in parallel for every subproblem.

% The matrices $S_k^\omega$ and $U_k^\omega$ map the different system components into the set of constraints that are directly relevant for subproblem $k$, where $\omega$ controls the size of neighborhood each node [gets information from]. 

The local step~\eqref{eq:iteration} can be interpreted as a \textit{block-Jacobi} iteration. In other words, it consists of solving the coupling system~\eqref{eq:coupling_gradients} only for the constraints in $\V_k^\omega$ while keeping the estimates of the coupling Jacobian constant for other constraints, and then as mapping the final result to $\V_k^0$. We will show in the next section that the local projections $\lhs_k^\omega$ are indeed invertible under mild assumptions, guaranteeing that~\eqref{eq:Sk_def} and~\eqref{eq:Uk_def} are well defined. 

At this point, different subproblems may have different estimates of the coupling Jacobian corresponding to the same constraints. The results of iteration~\eqref{eq:iteration} thus need to be aggregated for every constraint. This is performed in three steps.
First, nodes send their current estimates of the coupling Jacobian to nodes involved in the same constraint, i.e. to nodes that are two hops away in the bipartite graph $\G_b$. Then, coupling Jacobian estimates associated with a given constraint are averaged locally so that it is the same across all nodes involved in that constraint. Finally, results are broadcast throughout the network so that each node has a local copy of the entire vector $\hat{y}^{(t)}$.
The resulting scheme is the same as the \textit{synchronous coordination} algorithm in~\cite{Shin2020-jc} except for the presence of the aggregation step via message passing.

From a global perspective, the entire scheme can be written compactly as
\begin{align}\label{eq:sync_coord}
    \hat{y}^{(t+1)} = \Gamma(S^\omega \hat{y}^{(t)} + U^\omega q),
\end{align}
where the aggregation matrix $\Gamma \in \R^{\Lambda \times \sum_k |\V_k^\omega|}$ simply averages the results over the nodes involved in each constraint\footnote{A formal definition of $\Gamma$ is deferred to the Appendix~\ref{app:Gamma} for notational simplicity.}
% \begin{equation}
%     \Gamma_{i, j} =
%     \begin{cases}
%     \frac{1}{\delta_i} & \text{ if } j \in \V^0_i,\\
%     0 & \text{ otherwise,}
%     \end{cases}  
% \end{equation}
and where we define the matrices $S^\omega, U^\omega$ as $S^\omega =
    \begin{bmatrix}
        (S^\omega_1)^T, ..., (S^\omega_N)^T
    \end{bmatrix}^T,\
    U^\omega = \begin{bmatrix}
        (U^\omega_1)^T, ..., (U^\omega_N)^T
    \end{bmatrix}^T.$
% The iterative scheme in~\eqref{eq:sync_coord} is the same as in~\cite{Shin2020-jc}, except for the presence of the aggregation matrix $\Gamma$.

\subsection{Theoretical Guarantees}
% this is the important part: why do we have benefits?
% why does locality mean you have efficient approximations

% Now that we have presented the proposed scheme, we are in a position to discuss its theoretical guarantees.
% Chiefly, well-posedness of the update equations and the convergence rate of the overall algorithm.
We now turn to the theoretical properties of the proposed scheme, namely the well-posedness of the update equations and the convergence rate of the overall algorithm.

For the iteration in~\eqref{eq:iteration} to be well-defined, the local projections $\lhs_k^\omega$ need to be invertible.
Note that, if $\partial C$ is positive definite, its local projections are automatically positive definite and therefore invertible~\cite{Horn1990-ea}. However, in general, Jacobians of KKT matrices are \textit{not} positive definite. Indeed, in the presence of inequality constraints, the complementary slackness conditions imply that $\partial C$ is not even symmetric.
% In this section, we will first show that the local projections are nonetheless invertible under mild conditions. 
% The following result nonetheless guarantees invertibility .
With the following result, we show that these local projections are nonetheless invertible under mild conditions.

\begin{lemma}
    \label{lemma:Lkw_invertible}
    Consider a constraint-coupled problem in the form of Problem~\ref{pbm:CCOP} and an associated optimal solution $z^\star = (x^\star, \lambda^\star, \nu^\star)$. If the problem is strictly convex and satisfies Assumptions~\ref{A1},~\ref{A2},~\ref{A4} \&~\ref{A5}, then the projections $\lhs_k^\omega$ are invertible for all $k \in \V_p, \ \omega \in \mathbbm{Z}$.
\end{lemma}
% proof sketch
% steps in the proof:
% - 
The proof is provided in Appendix~\ref{proof:Lkw_invertible}. 
% It consists in showing that projections $\lhs_k^\omega$ are similar to an upper triangular block matrix whose diagonal blocks are full rank. We show they are full rank by showing that they are projections of matrices which are either positive definite or diagonal with nonzero diagonal terms. 
The significance of this result is that the iteration in Equation~\eqref{eq:iteration} is well-defined and applicable to our problem setting.
Under these conditions, the following result proves the convergence of the proposed algorithm.

\begin{theorem}[Adapted from{~\cite[Theorem 2]{Shin2020-jc}}]\label{th:convergence_rate}
Consider~\eqref{eq:coupling_gradients} and the decentralized scheme in~\eqref{eq:sync_coord} for a given set of node-associated constraint sets $\{\V_k^\omega\}_{k\in \V_p}$ and $\omega \geq 0$. Assume that the conditions of Lemma~\ref{lemma:Lkw_invertible} are verified. Then the sequence generated by~\eqref{eq:sync_coord} satisfies
\begin{align}
    \|y^{(t)} - y^\star \|_{\infty} \leq \alpha^t \|y^{(0)} - y^\star \|_\infty,
\end{align}
where the convergence rate $\alpha$ is defined by
\begin{align}
    \label{eq:alpha}
   \alpha := \max_{k \in \V_p} \frac{R_k \bar{\sigma}_k}{\ubar{\sigma}^2_k}  
   \left(
   \frac{\bar{\sigma}^2_k - \ubar{\sigma}^2_k}{\bar{\sigma}_k^2 + \ubar{\sigma}_k^2}
   \right)^{\left\lceil\frac{\omega}{2B_{\lhs_k^\omega}} - 1\right\rceil_+},
\end{align}
with $\bar{\sigma}^2_k$ and $\ubar{\sigma}^2_k$ the largest and smallest singular values of $\lhs_k^\omega$, respectively,
$R_k := \sum_{i \in \V_k^\omega, j \in \V \setminus \V_k^\omega} |(\lhs^\omega_{-k})_{i,j}|$ and $B_{\partial C_k^\omega}$ the bandwidth of matrix $\partial C_k^\omega$ induced by the bipartite graph\footnote{See Definition 3.3 in~\cite{Shin2022-jo}.}.
\end{theorem}
% This theorem implies that the convergence rate is dictated by the square of the condition number of each of the local projections.
The proof is provided in Appendix~\ref{proof:convergence_rate}. 
In contrast to \cite[Theorem 2]{Shin2020-jc}, Theorem~\ref{th:convergence_rate} does not require the system to be positive definite.
The main difference between the proofs is that the result in this paper relies on a generic theorem in~\cite[Theorem 3.6]{Shin2022-jo} that extends exponential decay of sensitivities to non-positive definite systems. 

% Separately, we show that the coupling system is actually banded (i.e. the exponential decay of sensitivities applies). Finally, we combine those results in Appendix~\ref{} to show that the result converges.
% The method exposed in~\cite{Shin2020-jc} initially relies on the matrix to be inverted to be Positive Definite (PD). The coupling Jacobian of interest here is not necessarily PD because it is not symmetric if there are coupling inequality constraints. However, we show that the methodology readily applies to our system thanks to Theorem~\ref{th:inverse_structure} and the specific structure of the coupling Jacobian (see Lemmas~\ref{lemma:banded} \&~\ref{lemma:Lkw_invertible} below).

The consequence of Theorem \ref{th:convergence_rate} is that the sequence~\eqref{eq:sync_coord} converges linearly to a solution of Equation~\eqref{eq:coupling_gradients}.
% Therefore, the sequence~\eqref{eq:sync_coord} defines a scheme with linear convergence to solve~\eqref{eq:coupling_gradients}. 
The convergence rate, $\alpha$, explicitly decreases exponentially with $\omega$, and further depends implicitly on $\omega$ through the conditioning and graph-induced structure of the local projections.
% We note that for a given value of $\omega$, the conditioning as well as the graph-induced structure of the system will also impact $\alpha$. 
While larger values of $\omega$ will decrease $\alpha$, it comes at a cost of more expensive communication, as more data needs to be exchanged both for constructing the projection matrices $\lhs_k^\omega, \lhs_{-k}^\omega, q_k^\omega$ and synchronizing the updates at every iteration. The computational cost of one iteration of~\eqref{eq:sync_coord} also increases with $\omega$ due to the increased size of  $S_k^\omega$.
% We note, however, that $S_k^\omega$ remains unchanged with iterations, so the factorization of $\lhs_k^\omega$ can be stored and does not need to be recomputed. 

Equation~\eqref{eq:alpha} suggests that the algorithm converges rapidly if $\omega$ is large compared to the bandwidth $B_{\lhs_k^\omega}$ of the local projections. Therefore, it is particularly effective for problems that have a naturally small bandwidth, i.e. highly structured problems.
% As the bandwidth quantifies the degree of connectivity of the underlying network, this criterion implies that enough information between neigh
In Lemma~\ref{lemma:banded} in Appendix~\ref{proof:banded}, we show that the structure of the coupling Jacobian is directly determined by that of the coupling constraints. Indeed, $\lhs$ can be interpreted as the Lagrangian of the graph underlying the coupling constraints.

\section{Numerical Experiments}\label{sec:experiments}
In this section, we validate our claims on the complexity and convergence rate of the proposed scheme.
% report experimental results testing two features of the differentiation scheme proposed in this paper.

\subsection{Validation of the complexity model}\label{sec:validation}

\begin{figure*}[ht!]
    \centering
    \begin{subfigure}[t]{0.48\textwidth}
    \centering
    \includegraphics[width=.8\textwidth]{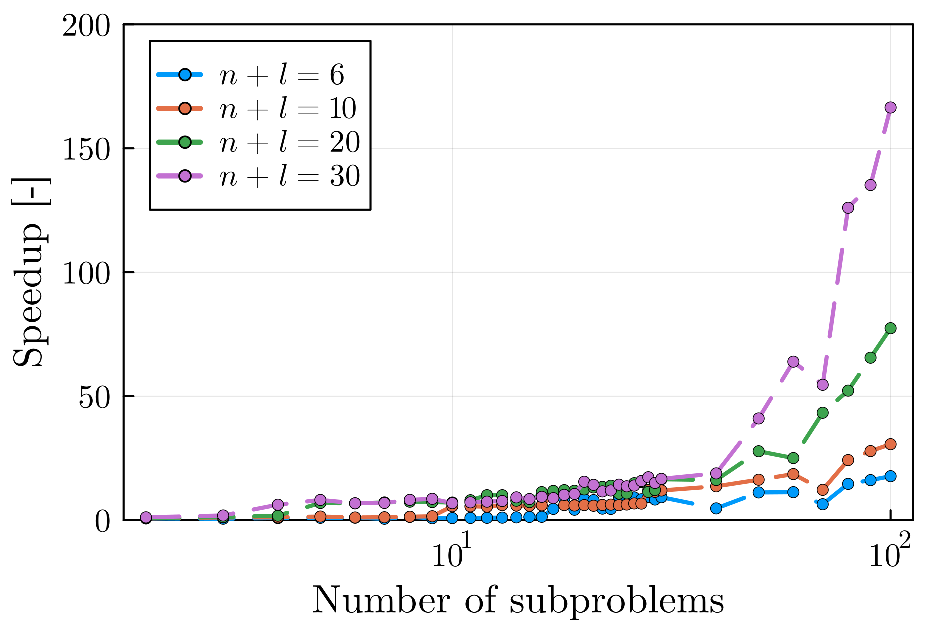}
    \caption{Speedup as a function of number of subproblems, for a fixed number of coupling constraints $\Lambda=2$. Each curve represents a problem with a different subproblem size. Each point is the ratio of the minimum compute times over 20 trials for each method.}
    \label{fig:scaling}
    \end{subfigure}
    \hfill
    \begin{subfigure}[t]{0.48\textwidth}
     \centering
     \includegraphics[width=.8\textwidth]{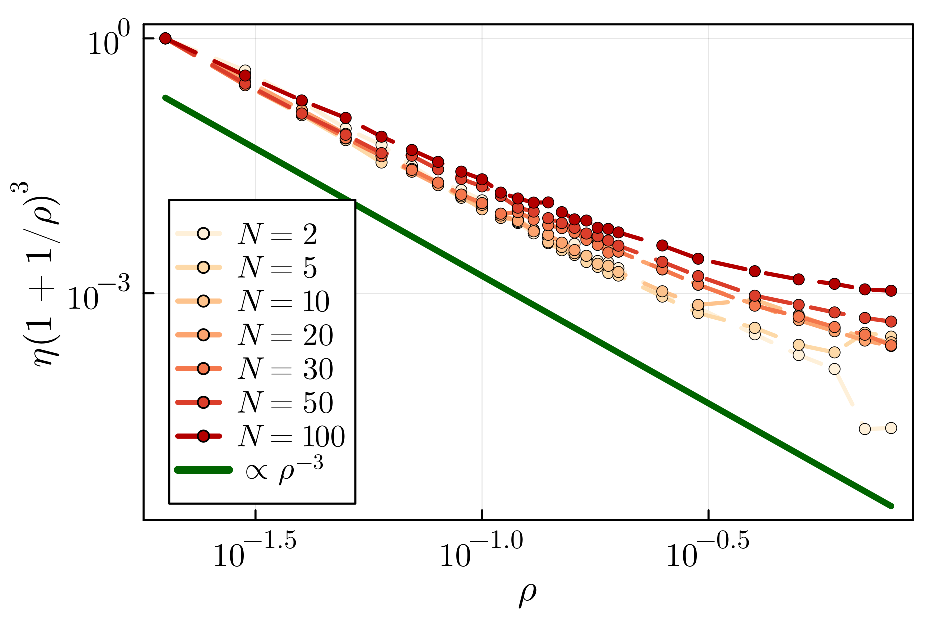}
     \caption{Rescaled ratio of compute times, for different numbers of subproblems while keeping $\rho$ constant. For small $\rho$, the scaling is as $\rho^{-3}$. All curves are normalized by their value at $\rho = 0.02$ to extract the scaling behavior.}
    \label{fig:scaling_rho_fixed}
\end{subfigure}
\caption{Numerical experiments testing the complexity model of~\eqref{eq:complexity}. The decentralized scheme delivers significant computational gains in comparison to centralized differentiation, especially for weakly coupled problems (i.e. small $\rho$).}
\vspace{-10pt}
\end{figure*}

Two different experiments were designed to test the simple scaling model of Eq.~\eqref{eq:complexity}. For these experiments, data are generated at random according to input dimensions, and each subproblem is given the same size. 

In the first experiment, the number of coupling constraints $\Lambda$ and the size of local problems $n+l$ are kept constant and only the number of subproblems $N$ is increased (see Fig.~\ref{fig:scaling}). This results in a speedup ($\eta^{-1}$) that increases rapidly with $N$ and with subproblem size -- this is consistent with Eq.~\eqref{eq:complexity}. 
% \footnote{We note that if the central Jacobian is represented as a sparse matrix, the qualitative features of the scaling are conserved but the speedup increases more slowly (results not reported here).} 
%In the remainder of this text, the central Jacobian is always represented as a sparse matrix.

% \begin{figure}[ht]
%     \centering
%     \begin{subfigure}[t]{0.45\textwidth}
%          \centering
%         \includegraphics[width=\textwidth]{Figures/Scaling/dense.pdf}
%     \caption{Central Jacobian represented as a dense matrix.}
%     \label{fig:scaling_dense}
%      \end{subfigure}
%     \hfill
%      \begin{subfigure}[t]{0.45\textwidth}
%          \centering
%          \includegraphics[width=\textwidth]{Figures/Scaling/sparse.pdf}
%          \caption{Central Jacobian represented as a sparse matrix.}
%          \label{fig:scaling_sparse}
%      \end{subfigure}
%       \caption{Speedup as a function of number of subproblems, for a fixed number of coupling constraints $\Lambda=2$. Each curve represents a problem with a different subproblem size. Each point is the ratio of the minimum compute times over 20 trials for each method.}
%     \label{fig:scaling}
% \end{figure}

In the second experiment, we keep the coupling ratio $\rho$ constant while increasing the number of subproblems $N$ by adjusting the number of coupling constraints $\Lambda$. The size of the local problems $(n+l)$ is kept fixed. Results are reported in Fig.~\ref{fig:scaling_rho_fixed}. 
% Similarly as above, the speedup increases rapidly as the number of problems is increased (see Fig.~\ref{fig:scaling_rho_fixed}). However, it does 
In this case, we do not expect an exponential increase in speedup because the number of coupling constraints itself increases. According to Eq.~\eqref{eq:complexity}, when $\rho$ is small enough, $\eta(1+1/\rho)^3 \sim \rho^{-3}$. This expectation is validated in Fig.~\ref{fig:scaling_rho_fixed}. We also observe a deviation from this scaling as $N$ increases. This is consistent with the expectation that $\eta \to (1 + 1/\rho)^{-3}$ as $N \to \infty$, and is due to the increasing relative weight of local gradient computation.

% \begin{figure}[ht]
%      \centering
%      \begin{subfigure}[t]{0.45\textwidth}
%          \centering
%          \includegraphics[width=\textwidth]{Figures/Scaling/ρ_fixed_th0.8_speedup.pdf}
%          \caption{Speedup as a function of the number of subproblems, for different values of the coupling ratio $\rho$.}
%          \label{fig:fixed_rho}
%      \end{subfigure}
%      \hfill
%      \begin{subfigure}[t]{0.45\textwidth}
%          \centering
%          \includegraphics[width=\textwidth]{Figures/Scaling/ρ_fixed_2.pdf}
%          \caption{Rescaled ratio of compute times, for different numbers of subproblems. For small $\rho$, the scaling is is as $\rho^{-3}$. All curves are normalized by their value at $\rho = 0.02$ to extract the scaling behavior.}
%          \label{fig:fixed_rho_rescaled}
%      \end{subfigure}
%     \caption{Scaling of the computational gain for fixed coupling ratio $\rho$. The subproblem size $n+l$ is fixed at 30, and the number of coupling constraints $\Lambda$ is increased as $N$ increases to keep a constant $\rho$.}
%     \label{fig:scaling_rho_fixed}
% \end{figure}

These experiments are in agreement with the scaling model of Eq.~\eqref{eq:complexity}. The computational gains provided by the distributed framework in this paper are significant, especially when the number of subproblems is large and when the coupling ratio $\rho$ is small.

\subsection{Distributed gradient computation}\label{sec:distributed_gradients}

\begin{figure*}[h]
    \centering
    \begin{subfigure}[t]{0.48\textwidth}
         \centering
        \includegraphics[width=.8\textwidth]{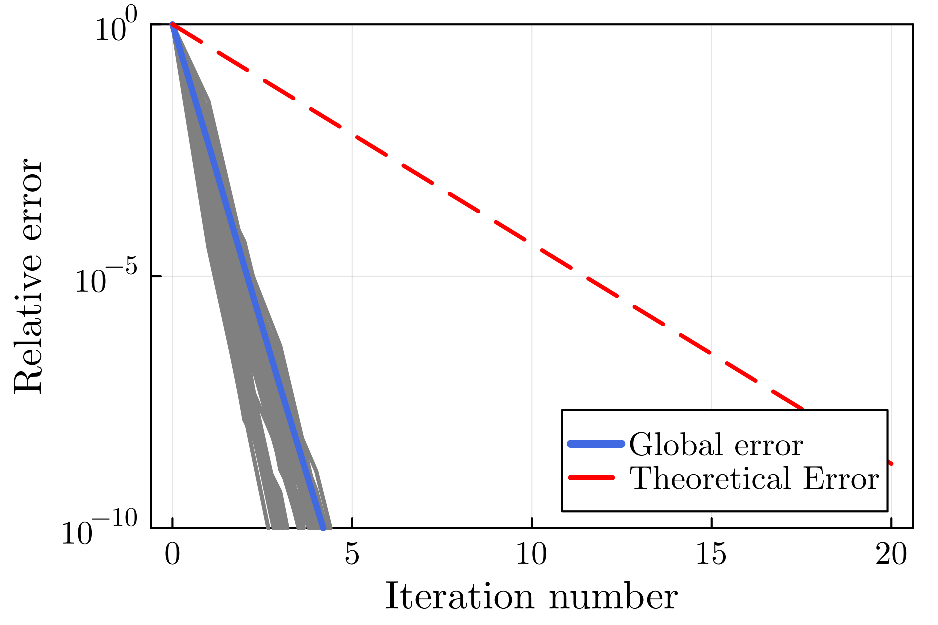}
    \caption{Relative error as a function of the iteration number for $\omega=1$. Gray lines: relative error for $\hat{y}_k$. Blue line: relative error for the full coupling Jacobian (i.e. $y^{(t)}$ in~\eqref{eq:sync_coord}).}
    \label{fig:approx_1}
     \end{subfigure}
    \hfill
    % \begin{subfigure}[t]{0.45\textwidth}
    %      \centering
    %     \includegraphics[width=.9\textwidth]{Figures/Approx/ω3_B1_N50_l2_k2_ΛM49_ΛF0.pdf}
    % \caption{Relative error as a function of the iteration number for $\omega=3$.}
    % \label{fig:approx_2}
    %  \end{subfigure}
    \begin{subfigure}[t]{0.48\textwidth}
    \centering \includegraphics[width=.8\textwidth]{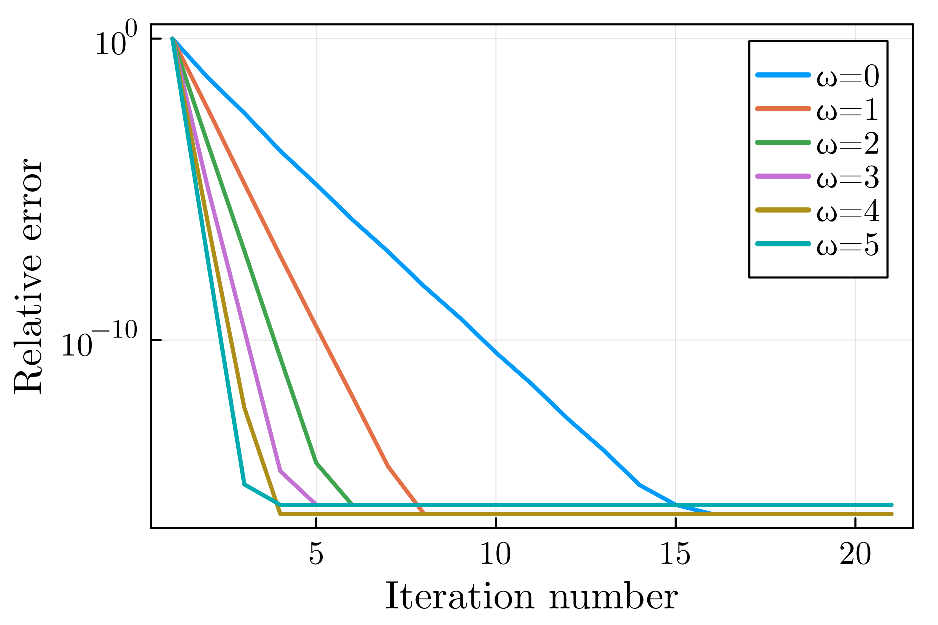}
    \caption{Relative error as a function of the iteration number for different values of the overlap parameter $\omega$.}
    \label{fig:conv_rate_omega}
    \end{subfigure}
     \caption{Relative error between the coupling Jacobian estimate and its true value for a randomly generated problem with the following parameters: each of the $N=50$ subproblems is constrained by $l=2$ inequality and $k=2$ equality constraints, and is involved in 2 equality coupling constraints. The distributed scheme converges linearly, at a rate that increases with $\omega$.
     }
    \label{fig:approx}
\end{figure*}

% Results from Section~\ref{sec:distributed} suggest that the fully distributed scheme is applicable to our context, i.e. for the computation of the coupling Jacobian.
The empirical convergence properties of the distributed scheme on a randomly generated problem are illustrated in Fig.~\ref{fig:approx}. In Fig.~\ref{fig:approx_1}, we report the relative $\infty$-norm error between the estimate of the coupling Jacobian and its true value, as a function of the number of iterations for $\omega=1$. 
% Results are reported for two different values of the horizon parameter $\omega$, along with 
We also report the theoretical upper bound on the convergence rate from Theorem~\ref{th:convergence_rate}.
We see that linear convergence is observed for every subproblem. We note that convergence is much faster than the bound indicates, hinting at potential for improvement. Empirical errors for different values of $\omega$ are reported in Fig.~\ref{fig:conv_rate_omega}.  
% With neighborhood size $\omega=3$, we observe significantly faster convergence than with $\omega=1$, as expected. 
From Theorem~\ref{th:convergence_rate}, we expect the convergence rate to increase with $\omega$, which is seen clearly in Fig.~\ref{fig:conv_rate_omega}.

% The scheme in~\eqref{eq:sync_coord} is therefore appropriate in the context of structured coupling equality and inequality constraints. The parameter $\omega$ can be set to achieve desired convergence rates. However, this comes at the expense of more communication overhead to construct the local projections as well as in more local computation at each iteration. This trade-off needs to be taken into account when selecting parameter values.
% \input{body/5.Applications}
\section{Rigorous approximation of Marginal Emissions in dynamic power systems models}\label{sec:MEF_approx}

Dynamic problems, i.e. optimization problems solved over time, are one general problem class that can benefit substantially from the proposed approach. Indeed, a time horizon corresponds to a linear graph where each node represents a snapshot in time of the system at hand. Each node is only connected to the timesteps immediately preceeding or succeeding it, resulting in a constant bandwidth of 1 regardless of the time horizon the problem is solved over. As shown in the previous section, a good approximation of gradients can be obtained by considering only a small and localized window centered at each time $t$. 
% The benefits of the approach developed in this paper are straightforward in cases where gradient computation is an end goal.

In dynamic power system models, accurate quantification of the emissions impacts of electricity demand relies on the ability to compute gradients efficiently. In particular, \textit{marginal emissions rates} are defined as the sensitivity of emissions to electricity demand in power systems~\cite{Siler-Evans2012-ve,Rudkevich2012-tx} and are an important emissions metric in the context of policies that aim to shift electricity demand or supply. 
Recent work has demonstrated the applicability of \textit{centralized} differentiable optimization techniques to compute marginal emissions rates in power system models~\cite{Fuentes_Valenzuela2023-dj}. However, those methodologies do not scale well to large networks solved over long time horizons.
Rather than computing those emissions rates jointly for the entire time horizon, the results in this paper justify computing an approximation relying only on local information.

% Exponential decay of sensitivity suggests that the approximation error decreases exponentially with the size of the window.
% The results in this paper allow us to rigorously estimate marginal emissions in dynamic power systems models over long time horizons in a distributed manner. 
% Given that dynamic problems can be represented as a linear graph, exponential decay of sensitivity suggests the error made by considering a window centered at time $t$ decreases exponentially with the size of the window. 
% We can accurately quantify the relevant time horizon using the proposed methodology.

In this section, we apply one iteration of the algorithm in~\eqref{eq:sync_coord} for different values of the horizon parameter $\omega$ on network models of the U.S.\ Western Interconnect. 
We model the power system under the DC-OPF approximation~\cite{Fuentes_Valenzuela2023-dj}%~\cite{Fuentes_Valenzuela_undated-vl}
, with network data generated using the PyPSA USA model.
The nodes in the network can be clustered to between 30 and 4786 nodes as needed.
% ; clustering is performed using nodal renewable profiles.
In our experiments, we 
% cluster the network model to 50 nodes and 
configure the number of batteries and storage penetration, adding batteries to the nodes with the highest peak renewable generation. In particular, we consider three different levels of storage penetration, defined as a function of the average daily load: 1\%, 10\% and 25\%. We solve each problem for a total horizon of 120 hours, and analyze the marginal emissions estimates. We report results for a 50-node aggregation in Fig.~\ref{fig:approx_MEF}. Similar results were obtained with both smaller and larger networks.

As expected, the marginal emissions approximation error decreases exponentially with the horizon parameter $\omega$ for all three scenarios, albeit at different rates. Indeed, the time horizon necessary to achieve a given level of accuracy increases with storage penetration. As storage capacity increases, batteries have an increasing ability to act as buffers over long time horizons. These results clearly illustrate the \textit{dynamic}, i.e. time-dependent, nature of marginal emissions rates. It is consistent with previous work demonstrating that a static approximation is often not appropriate~\cite{Fuentes_Valenzuela2023-dj}. 
The methodology developed in this paper thus offers a rigorous and efficient approach to estimate marginal emissions in a dynamic setting without using all of the problem data. 

% to quantifying how far in time these sensitivity metrics propagate. When computing marginal emissions rates in models solved over long time-horizons, one may thus not need to use all the problem information as their their propagation is limited in time. 

\begin{figure}
    \centering
    \includegraphics[width=.38\textwidth]{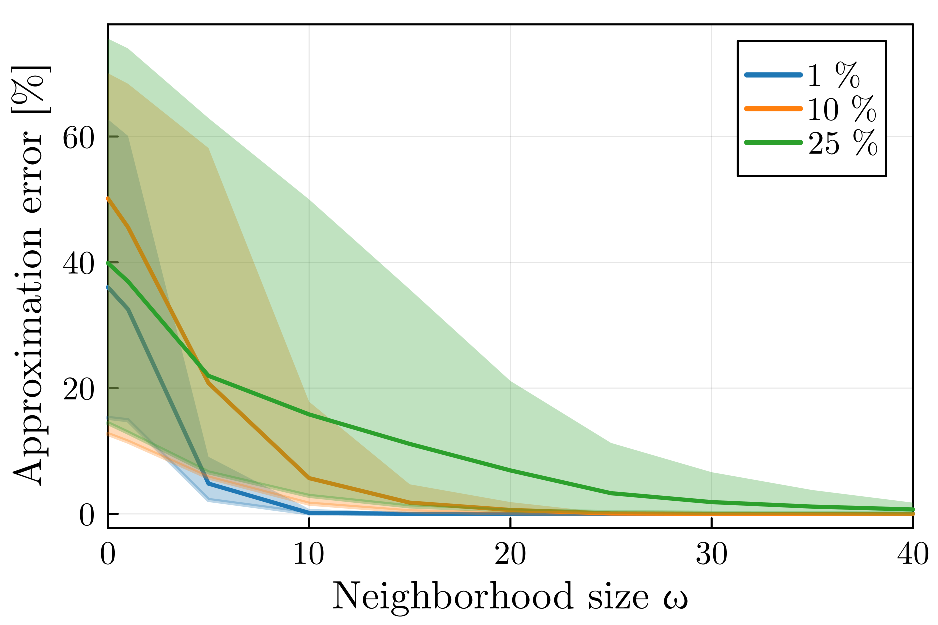}
    \caption{Approximation error of the marginal emissions rates as a function of the horizon parameter $\omega$, under different storage penetration scenarios for a 50-node network.
    Full line: median over all nodes in the network. Ribbons indicate the 10th and 90th percentiles. The approximation error decreases exponentially with $\omega$, at a rate that depends on storage penetration.}
    \label{fig:approx_MEF}
    \vspace{-10pt}
\end{figure}
\section{Conclusion}\label{sec:conclusion}

In this paper, we propose a framework for efficiently differentiating through separable optimization programs. We split the computation of Jacobians between a local component and a correction term resulting from the coupling between subproblems. This approach is computationally efficient, especially if the number of coupling constraints is relatively small. We also provide sufficient conditions for all Jacobians to be well defined. In addition, we generalize the analysis of a recent decentralized algorithm and show that it can be readily applied to compute the coupling term in a fully distributed manner, with theoretical guarantees. Finally, we apply the distributed algorithm to the computation of marginal emissions rates in time-coupled power system models. We observe that the larger the storage penetration the further in time emissions impacts can be propagated, which stresses the importance of incorporating dynamic effects in the computation of marginal emissions rates.

% We see several extensions for future work. First, we are particularly interested in extending the breadth of problems that can be tackled under the framework of decentralized bi-level optimization. 

Combined with fully distributed solution methods for the lower-level problems, the approach in this paper opens the door to gradient-based distributed bi-level optimization. This has the potential to accelerate complex bi-level problems like expansion planning~\cite{Degleris2021-vb}. We see the efficient and scalable implementation of such methodologies as well as the study of their convergence properties as exciting avenues for future work.
% The approach presented in this paper allows for differentiation through bi-level problems with arbitrary constraints.

% First, we have assumed the presence of a central node able to gather local gradients and compute the coupling Jacobian. This \textit{federated} setting, while realistic for some problem instances, is not completely general. Performing this computation in a \textit{fully} distributed setting, and exploring the computational efficiency in that context, is of significant interest. Similarly, schemes for \textit{approximately} computing the coupling Jacobian, e.g. via Neumann series, could prove beneficial in strongly-coupled systems. 

%discuss separability, self dual homogeneous embeddings, etc.

% The benefits of the distributed approach to gradient computation is three fold

% (i) in distributed systems (e.g. privacy): it is already the case and the problem did not really have a solution before

% (ii) computational efficiency: many local node computations of smaller sizes

% (iii) quantified error

\vspace{-10pt}
\renewcommand{\baselinestretch}{0.9}
{
% \scriptsize
\printbibliography
}
% \bibliographystyle{IEEEtran}
% {\small
% \bibliography{paperpile.bib}
% }
% \printbibliography

\appendix
\section{Proofs}\label{app:proof}
The proofs below are often a result of the structure of the KKT matrix for generic optimization problems. Let us consider problems of the form
\begin{align}
\begin{array}{ll}
      & \min_x f(x, \theta)\\
    \text{s.t. } & g_i(x, \theta) \leq 0,\ i=1, ..., l,\\
    & h_j(x, \theta) = 0, \ j=1,..., m,
\end{array}
\label{pbm:validity_P}
\end{align}
and assume we have a solution $x^\star(\theta_0) \in \mathbbm{R}^n$ with associated dual variables $\lambda^\star(\theta_0) \in \mathbbm{R}^l,\ \nu^\star(\theta_0) \in \mathbbm{R}^m$ for parameter values given by $\theta_0$. For a problem of the form of~\eqref{pbm:validity_P}, we will write the Jacobian of the KKT matrix as
\begin{align}
\label{eq:jac_KKT}
\partial_z G = 
    \scriptsize
    \begin{bmatrix}
        \nabla^2 f & \partial_x h & \partial_x g\\
        \partial_x h & 0 & 0 \\
        \diag(\lambda) \partial_x g & 0 & \diag(g),
    \end{bmatrix}
\end{align}
where $h = [h_1(x, \theta), ..., h_m(x, \theta)]^T$ and $g = [g_1(x, \theta),..., g_l(x, \theta)]^T$.

\subsection{Proof of Theorem~\ref{th:DID}}\label{app:proof_did}
\begin{proof}
 First, if Assumption~\ref{A3} holds for each subproblem, then it holds for the global problem. Let us denote by $\mathcal{Y}_i$ the set of all $y\in \mathbbm{R}^{n_i}$ that satisfy Assumption~\ref{A3} for subproblem $i$. The set $\mathcal{Y}_g$ satisfying the conditions for the global problem is such that $\mathcal{Y}_g \subseteq \mathcal{Y}_1 \times ... \times \mathcal{Y}_N$ because the global problem contains the global coupling constraints in addition to local constraints. Therefore, if that condition is satisfied for every subproblem, it is satisfied $\forall y \in \mathcal{Y}_g$ and consequently holds for the global problem.
    
   Second, if the global problem satisfies Assumptions~\ref{A1}-\ref{A5}, then the Jacobian of the solution is well defined according to Theorem 2.1 in~\cite{Fiacco1976-bt}. 
   
   Third, if each individual subproblem satisfies the second order conditions, then the local Jacobians are all unique and invertible. Indeed, if the first order conditions are satisfied for the global problem, they also are satisfied for all local problems individually and each subproblem then satisfies the conditions of Theorem 2.1 in~\cite{Fiacco1976-bt}. 
   
   Finally, we are left to show that the coupling Jacobian is invertible. To show this, we observe that the Jacobian of the full problem is a block matrix $J_f = \scriptsize \begin{bmatrix}
        A & B^T\\
        B & 0
    \end{bmatrix}$, where $A$ is a block diagonal matrix composed of local Jacobians, and that the coupling Jacobian is the Schur complement $J_f / A$ of $A$ in $J_f$\cite{Boyd2004-vf,Gallier2011-pr}. As $\det(J_f) = \det(A)\det(J_f/A)$, and as $J_f$ and $A$ are invertible, the coupling Jacobian $J_f/A$ is invertible. 
\end{proof}

\subsection{Structure of $\lhs$}\label{proof:banded}
For a problem in the form of Problem~\eqref{pbm:CCOP}, let us write $M_C = \begin{bmatrix}
    \partial_{\mathrm{x}} H^T & \partial_{\mathrm{x}} F^T
\end{bmatrix}^T,$ a matrix containing the gradients of the coupling constraints.
\begin{lemma}\label{lemma:banded} 
Consider the graph $\G_b$ and define the index sets $\mathcal{I} = \{I_i\}_{i \in \V_c}$ and $\mathcal{J} = \{J_i\}_{i \in \V_p}$ that partition the set of constraints $\{1,... \Lambda\}$ and the set of subproblems $\{1,... N\}$, respectively.
If $B_{M_C}$ denotes the bandwidth of $M_C$ induced by $(\G_b, \I, \J)$ and $B_{\lhs}$ the bandwidth of $\lhs$ induced by $(\G_b, \I, \I)$, then $B_{\lhs} \leq 2 B_{M_C}$.
\end{lemma}
\begin{proof}
We can rewrite the LHS of Eq.~\eqref{eq:coupling_gradients} as 
{\small
\begin{equation}
    \lhs = 
    \scriptsize
    - \begin{bmatrix}
        I & 0\\
        0 & \diag(\lambda)
    \end{bmatrix}
    M_{C}
    \bar{G} % find another notation
    M_{C}^T
    - \begin{bmatrix}
        0 & 0\\
        0 & \diag(F(x))
    \end{bmatrix},
    \label{eq:decomp}
\end{equation}
}
    
where $\bar{G}$ is a block diagonal matrix where the $i$-th block $\bar{G}_{ii}$ is $ (\partial_{z_i}G_i)^{-1}_{1:n_i, 1:n_i}$. 
    The block diagonal matrix $\bar{G}$ has bandwidth 0 induced by the graph $(\G, \mathcal{J}, \mathcal{J})$. Writing $B_{M_C}$ the bandwidth of $M_C$ induced by $(\G, \mathcal{I}, \mathcal{J})$, and 
    applying the bandwidth properties of graph-structured matrices (see Lemma 3.5 in~\cite{Shin2022-jo}), we have that $B_{M_C \bar{G} M_C^T} \leq B_{M_C} + 0 + B_{M_C^T} = 2B_{M_C}$ induced by $(\G, \mathcal{I}, \mathcal{I})$. As the remaining diagonal matrices have bandwidth zero induced by $(\G, \mathcal{I}, \mathcal{I})$, we have the result that $B_{\lhs} \leq 2B_{M_C}$. 
\end{proof}

\subsection{Definition of $\Gamma$}\label{app:Gamma}
% indexing function returns a subproblem k and a constraint index l
Define $\tilde{y} = [\hat{y}_1^T, ..., \hat{y}_N^T]^T$ the vector consisting of stacks of the local estimates of $\hat{y}$.
Define the index map $m(j) \in \R \to \R$ that maps an index $j$ of $\tilde{y}$ to the index of the corresponding coupling constraint $l \in \{1, ... \Lambda\}$. The aggregation matrix $\Gamma$ can be written as 
\begin{equation}
    \Gamma_{i, j} =
    \begin{cases}
    \frac{1}{\delta_{i}} & \text{ if } i = m(j),\\
    0 & \text{ otherwise,}
    \end{cases}  
\end{equation}
where $\delta_i$ denotes the degree of constraint $i$ in $\G_b$. Therefore, each row of $\Gamma$ collects the components of $\tilde{y}$ that correspond to the same constraint and averages them. We note that the rows of $\Gamma$ sum to 1, i.e. $\|\Gamma\|_\infty = 1$.
\subsection{Proof of Lemma~\ref{lemma:Lkw_invertible}}\label{proof:Lkw_invertible}
\begin{proof}
First, note that if the problem is strictly convex, it automatically satisfies Assumption~\ref{A3}.
Then, projecting~\eqref{eq:decomp} onto $\V_k^\omega$, we see that
\begin{align}
    \lhs_k^\omega &= -\diag(\lambda_k^\omega) 
    M_k^\omega 
    - \diag(\kappa_k^\omega),
\end{align}
where $M_k^\omega = (T_k^\omega)^T M_C \bar{G} M_C^T T_k^\omega$,  $\lambda_k^\omega \in \R^{|\V_k^\omega|}$ is the projection of the first diagonal matrix in~\eqref{eq:decomp} on $\V_k^\omega$, and $\kappa_k^\omega\in \R^{|\V_k^\omega|}$ is the projection of the second diagonal matrix in~\eqref{eq:decomp} on the same constraint set.
As $\bar{G}$ is PSD (see Lemma~\ref{lemma:inverse_KKT_PSD} below),  $M_C\bar{G}M_C^T$ is also PSD. 

In addition, $M_C\bar{G}M_C^T$ is invertible. 
The matrix $M_C\bar{G}M_C^T$ is actually the Schur complement of a modified global Jacobian. Indeed, let us construct the matrix
$
\partial\tilde{G} =
\scriptsize
\begin{bmatrix}
    G & B \\ 
    B^T & 0
\end{bmatrix}
$
where $G$ is a block-diagonal matrix with principal blocks being local jacobians $\partial_{z_i}G_i$ and where $B$ is a juxtaposition of $\partial_{x_i} H$ and $\partial_{x_i}F$ at the corresponding indices. We know by Theorem~\ref{th:DID} that the local Jacobians are invertible. Therefore the top left block $G$ is invertible. 
Considering all constraints in the problem are assumed linearly independent, including the coupling constraints, the matrix above is full row and column rank as long as the number of coupling constraints is smaller than the dimensions of $G$. Therefore it is invertible, and its Schur complement is invertible, too.

Therefore, $M_C\bar{G}M_C^T$ is PD, as is any of its principal submatrices~\cite{Horn1990-ea}, implying that $M_k^\omega$ is PD. 
Following a similar argument as in Lemma~\ref{lemma:inverse_KKT_PSD} below, the projection $\lhs_k^\omega$ is similar to an upper triangular block matrix such that
$$
P \lhs_k^\omega P^T = \scriptsize 
\begin{bmatrix}
    \lhs_1 & \lhs_2 \\
    0 & \diag(\lhs_3)
\end{bmatrix}.
$$
The block $\diag(\lhs_3)$ is full rank because all the terms in the diagonal are non-zero. The block $[\lhs_1 \ \lhs_2]$ is also full row-rank because it corresponds to the non-zero rows of $\diag(\lambda_k^\omega) M_k^\omega$, and $M_k^\omega$ is full rank because it is PD. Therefore, the block $\lhs_1$ is full rank. This implies that the matrix $\lhs_k^\omega$ is full rank and therefore invertible.
\end{proof}

\begin{lemma}
\label{lemma:inverse_KKT_PSD}
    Consider an optimization problem of the form~\eqref{pbm:validity_P} and an associated optimal solution $z^\star = (x^\star, \lambda^\star, \nu^\star)$. If the problem is strictly convex and satisfies Assumptions~\ref{A1},~\ref{A2},~\ref{A4} \&~\ref{A5}, then the $n \times n$ block of the inverse of the Jacobian corresponding to the primal variables, i.e.
    $
    (\partial_{z}G)^{-1}_{1:n, 1:n}
    $,
is PSD.
\end{lemma}
\begin{proof}
The Jacobian of the KKT matrix~\eqref{eq:jac_KKT} can be written in block form as
$
\partial_z G = 
\scriptsize
\begin{bmatrix}
    A & B \\
    C & D
\end{bmatrix},
$
where the blocks are 
$A = \nabla^2 L$, 
$B = [\partial_x h\ \partial_xg]$, 
$C = 
\diag(\tilde{\lambda})
B^T$, 
$D = 
\scriptsize
\begin{bmatrix}
    0 & 0\\
    0 & \diag(g)
\end{bmatrix}
$ and where $\tilde{\lambda} = [\underbrace{1, ...,  1}_{m \text{ times}}, \lambda^T]^T$.

As the problem is strictly convex by assumption, $A \succ 0$ and the Jacobian $\partial_z G$ is invertible according to Theorem 2.1 in~\cite{Fiacco1976-bt}. Therefore, the Schur complement of $A$, $S = D - CA^{-1}B$ is invertible and we can write~\cite{Gallier2011-pr}
{\small
\begin{align}
    (\partial_{z}G)^{-1}_{1:n, 1:n}
&= A^{-1} \left(I + B S^{-1} CA^{-1}\right).
\label{eq:inv_KKT_factors}
\end{align}
}%
Note that the rows of $C$ that are identically zero correspond to the non-tight inequality constraints. Assumption~\ref{A5} implies that the corresponding rows of $D$ are nonzero.
Therefore,
% if $l_0$ denotes the number of non-tight inequality constraints,
there exists a permutation matrix $P$ such that the Schur complement $S$ is similar to a matrix with structure
$
PSP^T = \bar{S} = 
\scriptsize
\begin{bmatrix}
    S_1 & S_2 \\
    0 & \diag(g_0)
\end{bmatrix},
$
with $g_0$ the non-zero elements of $D$.
The top-left block $S_1$ can be written as
$
S_1 = - \tilde{P}CA^{-1}B\tilde{P}^T,
$
where $\tilde{P} = 
\scriptsize\begin{bmatrix}
    I_{m+l_0} & 0
\end{bmatrix} P$ and $l_0$ is the number of non-tight inequality constraints. We know that $\bar{S}$ is invertible and that its inverse will have a similar structure, i.e.
$
\bar{S}^{-1} =
\scriptsize
\begin{bmatrix}
    S_1^{-1} & S_2' \\
    0 & \diag(1/g_0)
\end{bmatrix}.
$

% The second factor in~\eqref{eq:inv_KKT_factors} satisfies
% {
% \begin{align}
% \label{eq:idempotent1}
% \begin{array}{rl}
%     \left(I + B S^{-1}CA^{-1}\right)^2 & = I + 2B S^{-1}CA^{-1}\\
%     &+ B S^{-1}CA^{-1}BS^{-1}CA^{-1}.
% \end{array}
% \end{align}
% }
The second factor in~\eqref{eq:inv_KKT_factors} satisfies
$
\left(I + B S^{-1}CA^{-1}\right)^2  = I + 2B S^{-1}CA^{-1}
    + B S^{-1}CA^{-1}BS^{-1}CA^{-1}.
$
Under the same permutation, we have that 
$
PCA^{-1}BP^T
= -
\scriptsize
\begin{bmatrix}
    S_1 & S_2 \\
    0 & 0
\end{bmatrix}. 
$
Therefore,
{\small
\begin{align*}
PCA^{-1}BP^TPS^{-1}P^T = 
-
\scriptsize
\begin{bmatrix}
    I & 0 \\
    0 & 0
\end{bmatrix}. 
\end{align*}
}%
While this is not the identity matrix, the zero rows of this matrix correspond to the zero rows of $C$. This leads to
$
    CA^{-1}BS^{-1}C = -C.
$
We can thus write $\left(I + B S^{-1}CA^{-1}\right)^2 = I + B S^{-1}CA^{-1}$.
% {\small
% \begin{align}
% \label{eq:idempotent2}
%     \left(I + B S^{-1}CA^{-1}\right)^2 = I + B S^{-1}CA^{-1}.
% \end{align}
% }%
This shows that $I + B S^{-1}CA^{-1}$ is idempotent, implying all its eigenvalues are either 0 or 1~\cite{Horn1990-ea}, and therefore that it is PSD. 

The two factors in~\eqref{eq:inv_KKT_factors} are PSD matrices, as is their product if and only if their product is normal~\cite{Meenakshi1999-fq}. This is the case if the product is symmetric. 
Equation~\eqref{eq:inv_KKT_factors} amounts to $A^{-1} + A^{-1}B S^{-1}\diag(\tilde{\lambda})B^TA^{-1},$
% {\small
% \begin{align*}
% \label{eq:normal} 
% A^{-1} + A^{-1}B S^{-1}\diag(\tilde{\lambda})B^TA^{-1}.
% \end{align*}
% }%
which is symmetric if $S^{-1}\diag(\tilde{\lambda})$ is itself symmetric. We write
{\small
\begin{align*}
    S^{-1}\diag(\tilde{\lambda}) & = P^TPS^{-1}P^TP\diag(\tilde{\lambda})P^TP\\
    & = P^T\bar{S}^{-1}\diag(\bar{\tilde{\lambda}})P = \scriptsize \begin{bmatrix}
        S_1^{-1}\diag(\bar{\tilde{\lambda}}_+) &0\\
        0 & 0
    \end{bmatrix},
\end{align*}
}%
where $\bar{\tilde{\lambda}}$ is the extended dual variable $\tilde{\lambda}$ under permutation $P$, and $\bar{\tilde{\lambda}}_+$ is the restriction of $\bar{\tilde{\lambda}}$ to its nonzero elements.
We note that 
{\small
\begin{align*}
S_1^{-1}\diag(\bar{\tilde{\lambda}}_+) &=
-\left(\tilde{P}\diag(\tilde{\lambda})B^TA^{-1}B\tilde{P}^T\right)^{-1}\diag(\bar{\tilde{\lambda}}_+)\\
&=-\left(\diag(\bar{\tilde{\lambda}}_+)\tilde{P}B^TA^{-1}B\tilde{P}^T\right)^{-1}\diag(\bar{\tilde{\lambda}}_+)\\
& = - \left(\tilde{P}B^TA^{-1}B\tilde{P}^T
\right)^{-1}
\diag(\bar{\tilde{\lambda}}_+)^{-1}
\diag(\bar{\tilde{\lambda}}_+)\\
&= - \left(\tilde{P}B^TA^{-1}B\tilde{P}^T
\right)^{-1}.
\end{align*}
}%
The above matrix is symmetric, implying that~\eqref{eq:inv_KKT_factors} is symmetric. As~\eqref{eq:inv_KKT_factors} is also a product of two PSD matrices, then it is itself PSD~\cite{Meenakshi1999-fq}. 
\end{proof}

\subsection{Proof of Theorem~\ref{th:convergence_rate}}\label{proof:convergence_rate}
\begin{proof}
The proof of this result mainly consists of small modifications of proofs in~\cite{Shin2020-jc}.

First, for convergence, we need $\rho(\Gamma S^\omega) < 1$. As the spectral radius $\rho(\cdot)$ satisfies $\rho(\cdot) \leq \|\cdot \|_\infty$ ~\cite{Horn1990-ea} and the $\|\cdot\|_\infty$ is a matrix norm, it is submultiplicative and
$\rho(\Gamma S^\omega) \leq \| \Gamma S^\omega \|_\infty \leq \| \Gamma \|_{\infty} \|S^\omega\|_\infty \leq \| S \|_{\infty}$ as $\| \Gamma \|_{\infty} = 1$. Therefore, the same reasoning as in~\cite{Shin2020-jc} can be applied to our setting. Then, the proof of Lemma 3 in~\cite{Shin2020-jc} can be readily adapted with the result of Theorem 3.6 in~\cite{Shin2022-jo}. Indeed, the proof remains unchanged in spirit as long as the projections $\lhs_k^\omega$ are invertible, which we have proved in Lemma~\ref{lemma:Lkw_invertible}. Equality (48) in~\cite{Shin2020-jc} only needs to be replaced with the bound of Theorem 3.6 in~\cite{Shin2022-jo}.

\end{proof}

\end{document}